\documentclass[11pt]{article}

\usepackage[margin=1in]{geometry}
\usepackage{amsmath,amssymb,amsthm,mathtools}
\usepackage{bm}
\usepackage{booktabs}
\usepackage{graphicx}
\usepackage{hyperref}
\usepackage{xcolor}

\graphicspath{{figures/}{../figures/}}

\hypersetup{
  colorlinks=true,
  linkcolor=blue!55!black,
  citecolor=blue!55!black,
  urlcolor=blue!55!black
}

\newtheorem{proposition}{Proposition}
\newtheorem{lemma}{Lemma}
\newtheorem{remark}{Remark}

\newcommand{\OmegaD}{\Omega}
\newcommand{\bd}{\partial\Omega}
\newcommand{\ip}[2]{\left\langle #1,#2\right\rangle}
\newcommand{\norm}[1]{\left\lVert #1\right\rVert}
\newcommand{\R}{\mathbb{R}}

\title{Data generated internal solutions for the plasma wave equation: error bounds and numerical experiments in two dimensions}
\author{V.~Druskin\thanks{Department of Mathematical Sciences, Worcester Polytechnic Institute, Worcester, MA and Department of Mathematics, Southern Methodist University, Dallas, TX (\texttt{vdruskin1@gmail.com}).}
\and S.~Moskow\thanks{Department of Mathematics, Drexel University, Philadelphia, PA (\texttt{slm84@drexel.edu}).}
\and M.~Zaslavsky\thanks{Department of Mathematics, Southern Methodist University, Dallas, TX (\texttt{mzaslavskiy@smu.edu}).}}
\date{\today}

\begin{document}

\maketitle

\begin{abstract}
We consider the computation of internal solutions for a time domain plasma wave equation with an unknown potential $q$ from boundary response data.  The internal solutions are computed by transforming known background snapshots using the Cholesky decomposition of the data-driven Gramian, or mass matrix.  It was recently shown that in one dimension these data generated internal solutions converge in $L^2$ at order $\sqrt{\tau}$ for well chosen initial waves.  Here we study the internal solution reconstruction in two dimensions, where a multiple input/multiple output (MIMO) setup and a block Gramian are needed, with the number of boundary sources increasing as the time sampling is refined.  We show that the general error bound carries over to this setting: the distance between the data generated solutions and the best causal approximation from background snapshots is controlled by the best approximation mass matrix mismatch.  Numerical refinement studies on a square domain measure the convergence of the data generated solutions alongside the best causal approximation from the background snapshots, with the relative errors appearing to go to zero at rate $\sqrt{\tau}$, and the absolute errors going to zero in $L^1$. We also show that the solution reconstructions remain accurate for high contrast composite media.  Finally, since the mass matrix assembled from noisy response data can fail to be positive definite, regularization is needed; comparing a diagonal shift with an eigenvalue floor, we find the floor more robust, with reconstructions that remain more accurate than the unperturbed background field with noise that is up to ten percent of the root mean square data amplitude.
\end{abstract}

\section{Introduction}
\label{sec:intro}

In this paper we consider the computation of time domain wavefields in the interior of a medium that can be probed only from its boundary.  The model is the plasma wave equation
\begin{equation}
  \label{eq:plasma-intro}
  u_{tt}-\Delta u + q(x)u=0
\end{equation}
on a bounded domain, with the potential $q$ unknown, and we ask how accurately the interior snapshots generated by boundary-localized sources can be recovered from boundary response measurements alone.  The motivation comes from the inverse scattering problem of determining $q$ itself, which has an extensive literature, see for example \cite{CaCoMo,rakesh1988linearised,0266-5611-30-6-065005,bube1983one,Virieux2016AnIT}.  We consider a  closely related problem of recovering a Schr\"odinger potential from boundary data for multidimensional Schrodinger or plasma wave equations in overdetermined formulation, see \cite{novikov1988multidimensional}  and references wherein. This  overdetermined generalization of Marchenko-Gelfand-Levitan inverse spectral method, that was possibly introduced by Faddev in 1950s,  treats operator-valued spectral data, corresponding to $2d-1$ data sets in $d$-dimensional problems. Algorithms with access to interior fields can sidestep much of the nonlinearity of the inverse coefficient problem, and some approaches construct an approximate interior solution as an intermediate step for exactly this reason \cite{Kepley2016GeneratingVI}.

A particular route to interior data is provided by data-driven reduced order models (ROMs).  From evenly sampled boundary responses one can assemble, without knowing $q$, the exact Gramian of the unknown interior snapshots, and hence a small projected system consistent with the measurements; see \cite{doi:10.1137/1.9781611974829.ch7} for the model reduction framework and \cite{DrMaThZa,BoDrMaZa,BoDrMaZa2,BoDrMaZa3,DrMaZa} for inversion algorithms built on it.  The step at the center of the present paper was first introduced in \cite{BoDrMaMoZa}: transforming known background snapshots by the Cholesky factors of the measured Gramian yields remarkably accurate approximations of the unknown interior fields.  These data generated fields can then be inserted into the Lippmann--Schwinger integral equation to recover $q$ \cite{DrMoZa,DrMoZa2,DrMoZa3,DrMoZa4,BoGaMaZi}, and the quality of the resulting images is tied directly to their accuracy.  Much of this literature treats settings with additional difficulties: variable wave speeds  \cite{BoDrMaZa2,BoDrMaZa3,DrMaZa,BoGaMaZi}, diffusion problems \cite{baker2025regularized}, and monostatic or partially nonreciprocal data \cite{DrMoZa2,DrMoZa3,DrMoZa4}.  Here, as in \cite{DrMoZa6}, we deliberately work in the simplest setting that exhibits the field reconstruction step, the plasma wave equation with a full multiple input/multiple output array and only the potential unknown, so that the accuracy of the internal solutions can be studied without these complications.

In \cite{DrMoZa6}, we isolated this internal solution generation step for the plasma wave equation and studied it as a stand-alone problem with exact data.  There we showed that the data generated internal solutions are asymptotically close to the projection of the true solution onto the space of background snapshots, and we proved convergence in $L^2$ of order $\sqrt{\tau}$ in one dimension for two examples of initial waves.  As discussed in \cite{DrMoZa6}, in higher dimensions a single input/single output (SISO) background subspace will not generally be rich enough for convergence, and a multiple input/multiple output (MIMO) formulation is necessary, with the number of boundary sources increasing as the time sampling is refined: the time samples resolve the fields in depth, and the sources resolve them laterally along the boundary.  The purpose of the present paper is to formulate and test this MIMO reconstruction step in two dimensions, including for noisy data.  In exact arithmetic the reconstruction is a change of coordinates between two snapshot systems.  In computations, and especially with noisy data, the mass matrix may lose positive definiteness or become severely ill conditioned, so regularization is essential; regularized versions of this construction are known to be stable in practice \cite{BoDrMaZa3,baker2025regularized,DrMoZa3}.

In this work we state the formulation of the internal snapshot reconstruction for the time domain plasma wave equation, describe the block Gramian construction for multiple boundary sources and receivers, show that the general error bound of \cite{DrMoZa6} relating the data generated solutions to the best causal approximation carries over to this setting, and compare two regularizations of the noisy Gramian: a diagonal shift, which is a Tikhonov type regularization \cite{engl1996regularization}, and an eigenvalue floor, related to the nearest positive semidefinite matrix \cite{higham1988computing}.  The numerical experiments, performed with finite difference solvers, indicate that the eigenvalue floor stabilizes the reconstructions under response noise.  We also present two dimensional refinement studies motivated by the one dimensional $\sqrt{\tau}$ convergence rate of \cite{DrMoZa6}, together with a heuristic best approximation discussion of the background snapshot family that is consistent with the observed rates. The considered formulation is overdetermined -- we consider matrix-valued data corresponding to the $2d-1$ dimensional data sets in the $d$-dimensional continuous case. Our approach can be reduced to a non-overdetermined formulation, with only diagonals of the data matrix available in the framework of the data-completed Lippmann-Schwinger-Lanczos inversion algorithm \cite{DrMoZa4}.

The paper is organized as follows.  In section~\ref{sec:setup} we describe the forward model and the boundary data.  In section~\ref{sec:construction} we describe the mass matrix, its construction from the data, and the data generated internal solutions, including the block form needed for the MIMO setting.  In section~\ref{sec:bound} we state and prove a general error bound of \cite{DrMoZa6}, which carries over to this setting and bounds the distance between the data generated solutions and the best causal approximation of the true solutions by the mismatch of their mass matrices.  Section~\ref{sec:noise} contains the noisy data model and the two regularizations.  In section~\ref{sec:numerics} we show numerical experiments, and section~\ref{sec:discussion} contains a discussion of limitations and future work.

\section{Problem setup}
\label{sec:setup}

Let $\OmegaD\subset\R^d$, $d=1$ or $2$, be a bounded domain with boundary $\bd$.  We consider
\begin{equation}
  \label{eq:plasma-wave}
  \begin{cases}
    u_{tt}(x,t)-\Delta u(x,t)+q(x)u(x,t)=0,
      & x\in\OmegaD,\quad 0<t<T,\\
    {\partial u\over{\partial\nu}}= 0 ,
      & x\in\bd,\quad 0<t<T,\\
    u(x,0)=g_m(x),\qquad u_t(x,0)=0,
      & x\in\OmegaD,
  \end{cases}
\end{equation}
where $q$ is the unknown plasma potential and $m=1,\ldots,n_b$ indexes the incident field.  Here, as in \cite{DrMoZa6}, the source is implemented as a localized initial displacement near the boundary, with homogeneous Neumann boundary conditions.  This produces the same algebraic data structure as a boundary source after discretization: source and receiver functions are localized near $\bd$, and the measured response is an inner product against receiver profiles. Dirichlet boundary conditions or initial velocity conditions could be implemented similarly. 

For receiver profile $r_\ell$, source $g_m$, and sample times $t_j=j\tau$, define the response matrix samples
\begin{equation}
  \label{eq:response}
  F_{\ell m}(t_j)
    = \int_{\OmegaD} r_\ell(x)\,u_m(x,t_j)\,dx.
\end{equation}
The background medium is denoted by $q_0$, here $q_0=0$, and its corresponding snapshots are $u^0_m(\cdot,t_j)$.  These background snapshots are computable since $q_0$ is known.  The problem we consider in this manuscript is: given the response samples $\{F_{\ell m}(t_j)\}$, reconstruct the internal snapshots $u_m(\cdot,t_j)$ in the unknown medium.

\section{Construction of the internal solutions from the data}
\label{sec:construction}

Consider the true snapshots
\[
  U_{j,m}(x)=u_m(x,j\tau),\qquad
  U^0_{j,m}(x)=u^0_m(x,j\tau).
\]
An essential component of the construction is the Gramian, or mass matrix, of the snapshots.  It is well known that the mass matrix can be obtained from the boundary data; this construction is quite general for linear problems \cite{gosea2022data}.  For a SISO problem, $n_b=1$, and the cosine propagator identity gives
\begin{equation}
  \label{eq:siso-gramian}
  M_{ij}
  =
  \ip{U_i}{U_j}_{L^2(\OmegaD)}
  =
  \frac{1}{2}\left(F(|i-j|\tau)+F((i+j)\tau)\right),
\end{equation}
see \cite{DrMoZa6} and the references therein.  For multiple sources and receivers, the same identity is applied blockwise, see for example \cite{BoGaMaZi}.  Define the block Gramian
\begin{equation}
  \label{eq:block-gramian}
  M_{(i,\ell),(j,m)}
  =
  \ip{U_{i,\ell}}{U_{j,m}}_{L^2(\OmegaD)}
  =
  \frac{1}{2}\left(
    F_{\ell m}(|i-j|\tau)
    +
    F_{\ell m}((i+j)\tau)
  \right).
\end{equation}
Thus $M$ is obtained from measured response data alone.  Likewise, the background mass matrix $M_0$ is obtained either by direct quadrature of the background snapshots or by the same response identity in the known background.

\begin{remark}
Equation \eqref{eq:block-gramian} is the central identity that links the boundary data to the interior.  The measured traces do not directly reveal pointwise internal values, but they determine all pairwise inner products among the finite snapshot family.  Reconstruction is therefore a Gramian matching problem, and Proposition~\ref{prop:matching} below makes the characterization precise, as in \cite{DrMoZa6}.
\end{remark}

Let $U\in\R^{N\times P}$ be the matrix whose rows are the unknown snapshots $U_{j,m}$ sampled on the spatial grid, and let $U_0\in\R^{N\times P}$ contain the corresponding background snapshots.  Here $N=n_t n_b$, where $n_t$ is the number of time samples, $n_b$ is the number of boundary sources, and $P$ is the number of spatial grid points.  The rows are ordered \emph{time-major}: row $i=jn_b+m$ holds the snapshot of source $m$ at time step $j$, so the time index is the outer, slower one, and the $n_b$ snapshots sharing a time step form a consecutive block.  With $W$ the diagonal matrix of spatial quadrature weights, the mass matrices are
\begin{equation}
  \label{eq:mass-matrices}
  M = UWU^\top,\qquad M_0=U_0WU_0^\top .
\end{equation}

We remark on a notational difference from \cite{DrMoZa6}.  There the snapshots were treated as row vectors of functions in $L^2(\OmegaD)$ and inner products were written as integrals; here we work instead with vectors of point values on the spatial grid, and the weight matrix $W$ carries the integration, so that $UWU^\top$ is the quadrature approximation of the $L^2$ Gramian.  The two formulations are equivalent up to quadrature error, and all statements below hold verbatim in the function setting with $W$ absorbed into the inner product.  We prefer the discrete formulation here because it makes explicit each step of the algorithm: the mass matrix $M$ is assembled from the measured responses alone through \eqref{eq:block-gramian}, the background snapshots $U_0$ and the weights $W$ are computed from the known background medium, and the true snapshots $U$ never enter the reconstruction.  They appear only in the analysis and in the evaluation of errors in the synthetic experiments.

The exact mass matrices are Gramians of the snapshot families and are therefore symmetric positive definite whenever the families are linearly independent, which we assume throughout.  We compute their Cholesky decompositions
\begin{equation}
  \label{eq:cholesky}
  M=LL^\top,\qquad M_0=L_0L_0^\top;
\end{equation}
these decompositions were first used to solve the coefficient inverse problem in \cite{DrMaThZa}.  The data generated internal solutions used in this work are then
\begin{equation}
  \label{eq:reconstruction}
  \mathbf{U}
  =
  L L_0^{-1} U_0 ,
\end{equation}
or, equivalently, the numerically stable triangular-solve form used in the code,
\[
  \mathbf{U}
  =
  T^\top U_0,\qquad
  L_0^\top T=L^\top .
\]
The reconstructed row $\mathbf{U}_{j,m}$ approximates $u_m(\cdot,j\tau)$ in the unknown medium.

The matrix $M$ has a natural block structure: each time-time entry is a $K\times K$ source/receiver block, where $K=n_b$.  Therefore one could factor $M$ by block Cholesky, recursively forming and factoring the $K\times K$ Schur-complement pivots associated with each time step.  As discussed in \cite{DrMoZa6}, the block Cholesky decomposition involves a choice of square root of the positive definite diagonal blocks, and is therefore not unique.  In the computations reported here we just use the full Cholesky factorization of the regularized block matrix $M$, which is algebraically equivalent to choosing Cholesky as the square root in block Cholesky. 

\begin{proposition}[Exact matching]
\label{prop:matching}
The data generated family $\mathbf{U}$ in \eqref{eq:reconstruction} has the same mass matrix as the true snapshots,
\[
  \mathbf{U}W\mathbf{U}^\top=M,
\]
and it is the unique family of the form $SU_0$, with $S\in\R^{N\times N}$ lower triangular with positive diagonal entries, that has this property.
\end{proposition}

\begin{proof}
Using $U_0WU_0^\top=M_0=L_0L_0^\top$ and $\mathbf{U}=LL_0^{-1}U_0$, we obtain
\[
  \mathbf{U}W\mathbf{U}^\top
  =
  LL_0^{-1}M_0L_0^{-\top}L^\top
  =
  LL^\top
  =
  M.
\]
For the uniqueness, suppose $B=SU_0$ with $S$ lower triangular with positive diagonal entries satisfies $BWB^\top=M$.  Then $SL_0$ is lower triangular with positive diagonal entries and $(SL_0)(SL_0)^\top=SM_0S^\top=M$, so $SL_0=L$ by uniqueness of the Cholesky factorization, and $B=LL_0^{-1}U_0=\mathbf{U}$.
\end{proof}

In this proposition, the approximation enters through the choice of the background snapshot family as a coordinate system for the unknown snapshots.  In \cite{DrMoZa6} it was shown that the error in the data generated solutions is controlled by the error in the best approximation of the true snapshots from this space of background snapshots, and that these errors become asymptotically close when the best approximation error is small.  We state the corresponding bound for the present MIMO formulation in section~\ref{sec:bound}.

\section{A general error bound}
\label{sec:bound}

The general error analysis of \cite{DrMoZa6}, developed there for the SISO problem, carries over to the MIMO formulation of section~\ref{sec:construction}.  As observed in section 7 of \cite{DrMoZa6}, once the scalar rather than block Cholesky factorization is used, the arguments are purely algebraic and hold in any spatial dimension; only the time-major ordering of the snapshot family enters.  We do not reproduce the full set of results here, only the bound relating the data generated solutions to admissible approximations through their mass matrices.

For a snapshot family $V\in\R^{N\times P}$, stored by rows as the families in \eqref{eq:mass-matrices}, define
\begin{equation}
  \label{eq:stacked-norm}
  \norm{V}_2
  =
  \left(\operatorname{trace}\left(VWV^\top\right)\right)^{1/2}
  =
  \left(\sum_{i=0}^{N-1}\norm{v_i}_{L^2(\OmegaD)}^2\right)^{1/2},
\end{equation}
where $v_i$ are the rows of $V$.  Here and in what follows, with a slight abuse of notation, $L^2(\OmegaD)$ norms and inner products of grid functions denote the discrete ones induced by the quadrature weights $W$.  Following \cite{DrMoZa6}, we call an approximation $B\in\R^{N\times P}$ \emph{admissible} if $B=SU_0$ for some lower triangular matrix $S\in\R^{N\times N}$ with positive diagonal entries.  Here $S$ is the matrix of expansion coefficients of the rows of $B$ in the background snapshots: row $i$ of an admissible $B$ is $\sum_{k\le i}S_{ik}(U_0)_k$, a combination of background snapshots from the same or earlier positions in the time-major ordering, and the diagonal entry $S_{ii}$ is the coefficient of the background snapshot at position $i$ itself.  In the column notation of \cite{DrMoZa6}, $S$ corresponds to the transpose of the upper triangular matrix $\tilde{T}$ there.  The data generated family $\mathbf{U}=LL_0^{-1}U_0$ of \eqref{eq:reconstruction} is itself admissible, with coefficient matrix $S=LL_0^{-1}$, the transpose of the matrix $T=L_0^{-\top}L^\top$ of \cite{DrMoZa6}; as shown in the proof of Proposition~\ref{prop:general-bound} below, every admissible $B$ has coefficient matrix of this form, $S=\widehat{L}L_0^{-1}$, with $\widehat{L}$ the Cholesky factor of its own mass matrix.

Among admissible approximations there is a natural best one.  For each row index $i$, let $P_i$ denote the $W$-weighted $L^2(\OmegaD)$ orthogonal projection onto $\operatorname{span}\{(U_0)_0,\ldots,(U_0)_i\}$, the span of the background snapshots at the same or earlier positions, and define the family of \emph{sequential causal projections} of the true snapshots,
\begin{equation}
  \label{eq:causal-projection}
  \widehat{U}_i = P_i U_i,
  \qquad i=0,\ldots,N-1,
\end{equation}
where $U_i$ denotes row $i$ of the true family $U$ of \eqref{eq:mass-matrices}.  By construction $\widehat{U}=\widehat{S} U_0$ with $\widehat{S}$ lower triangular, and each row of $\widehat{U}$ is the closest element of its causal subspace, so if the diagonal entries of $\widehat{S}$ are positive, then $\widehat{U}$ is admissible and minimizes $\norm{U-B}_2$ over the admissible set; we then call $\widehat{U}$ the \emph{best admissible approximation} to $U$.  Positivity of the diagonal entries is not automatic and depends on the choice of sources and sampling; conditions ensuring it for $\tau$ small are given in Lemma~\ref{lem:positivity} at the end of this section.

\begin{proposition}
\label{prop:general-bound}
Let $U$, $U_0$, $W$, and the mass matrices $M$, $M_0$ be as in \eqref{eq:mass-matrices}, with Cholesky factors $L$, $L_0$ as in \eqref{eq:cholesky}, and let $\mathbf{U}=LL_0^{-1}U_0$ be the data generated family \eqref{eq:reconstruction}.  Let $\widehat{U}$ be the best admissible approximation \eqref{eq:causal-projection} to the true snapshots, and let $\widehat{M}=\widehat{U} W\widehat{U}^\top$ be its mass matrix.  Then for
\[
  \epsilon'
  =
  \kappa_2(M)\,\frac{\norm{M-\widehat{M}}_F}{\norm{M}_2},
  \qquad
  \kappa_2(M)=\norm{M}_2\,\norm{M^{-1}}_2,
\]
small enough,
\begin{equation}
  \label{eq:general-bound}
  \norm{\mathbf{U}-\widehat{U}}_2
  \le
  \epsilon'\,\norm{U}_2 .
\end{equation}
That is, the data generated internal solutions are within relative distance $\epsilon'$, in the stacked norm \eqref{eq:stacked-norm}, of the best causal approximation of the true solutions from the background snapshot family.
\end{proposition}

\begin{proof}
The proofs are the same as those of lemma 4.1 and proposition 4.3 of \cite{DrMoZa6}.  Let $\widehat{L}$ be the Cholesky factor of $\widehat{M}$.  We first show that
\begin{equation}
  \label{eq:norm-identity}
  \norm{\mathbf{U}-\widehat{U}}_2=\norm{L-\widehat{L}}_F,
\end{equation}
where $\norm{\cdot}_F$ denotes the Frobenius norm.  The matrix $\widehat{S} L_0$ is lower triangular with positive diagonal entries and satisfies
\[
  (\widehat{S} L_0)(\widehat{S} L_0)^\top=\widehat{S} M_0(\widehat{S})^\top=\widehat{U} W\widehat{U}^\top=\widehat{M},
\]
so $\widehat{S} L_0=\widehat{L}$ by uniqueness of the Cholesky factorization, and $\widehat{S}=\widehat{L}L_0^{-1}$ has the same Cholesky product form as the data generated family.  Then $\mathbf{U}-\widehat{U}=(L-\widehat{L})L_0^{-1}U_0$, and
\[
  (\mathbf{U}-\widehat{U})W(\mathbf{U}-\widehat{U})^\top
  =
  (L-\widehat{L})L_0^{-1}M_0L_0^{-\top}(L-\widehat{L})^\top
  =
  (L-\widehat{L})(L-\widehat{L})^\top,
\]
and taking traces of both sides gives \eqref{eq:norm-identity}.  We note also that $\norm{\mathbf{U}}_2^2=\operatorname{trace}(LL^\top)=\operatorname{trace}M=\norm{U}_2^2$ by Proposition~\ref{prop:matching}.  By the forward stability of the Cholesky factorization \cite{ChPaSt} (see also theorem 4.2 of \cite{DrMoZa6}), if $\epsilon=\norm{M-\widehat{M}}_F/\norm{M}_2$ satisfies $\epsilon\,\kappa_2(M)<1$, then
\[
  \norm{L-\widehat{L}}_F
  \le
  \frac{1}{\sqrt{2}}\,\norm{L}_2\,\kappa_2(M)\,\epsilon
  +
  O\!\left(\epsilon^2\right).
\]
The result follows from \eqref{eq:norm-identity} using $\norm{L}_2\le\norm{L}_F=\norm{U}_2$.
\end{proof}

\begin{remark}
The proof uses only the admissibility of $\widehat{U}$, so \eqref{eq:general-bound} holds verbatim with $\widehat{U}$ replaced by any admissible approximation and $\widehat{M}$ by its mass matrix.  We state the proposition for the best admissible approximation because of what it then says: whenever the mass matrix of the best causal approximation is close to the data Gramian and $M$ is well conditioned, the data generated solutions are asymptotically optimal in the admissible set. In computations with noisy data or modeling error $M$ may be indefinite; it is then symmetrized and regularized as in section~\ref{sec:noise}, and the identity \eqref{eq:norm-identity} holds with the factor actually used.
\end{remark}

We close this section with conditions ensuring that the sequential causal projections are admissible, implying that they are the best approximations referenced in Proposition~\ref{prop:general-bound}. For this we need that the diagonal entries of $\widehat{S}$ are positive.  Let $\{\bar u^0_i\}$ be the sequential orthonormalization, in the $L^2(\OmegaD)$ inner product and in the time-major ordering, of the background snapshots (the rows of $U_0$ in \eqref{eq:mass-matrices}), so that $(L_0)_{ii}=\ip{(U_0)_i}{\bar u^0_i}>0$.  The pivot $(L_0)_{ii}$ is the norm of the component of $(U_0)_i$ orthogonal to the preceding background snapshots, that is, its $L^2(\OmegaD)$ distance to their span; it equals $\norm{(U_0)_i}_{L^2(\OmegaD)}$ when the background snapshots are orthogonal, and is smaller otherwise.  The setting we have in mind is that of \cite{DrMoZa6} and of section~\ref{sec:numerics}: we have a family of problems where the sources are approximate delta functions of parameter $\tau$, normalized in $L^1(\OmegaD)$ ($g_m\ge0$, $\int_{\OmegaD}g_m\,dx=1$), with supports of diameter at most $2\tau$ contained near $\bd$ and disjoint from the support of $q$.  The lemma itself, however, assumes only what its conclusion needs.

\begin{lemma}[Positivity of the diagonal entries]
\label{lem:positivity}
Let $\widehat{U}= \widehat{S} U_0$ be the sequential causal projections of the true snapshots for a family of problems indexed by $\tau$. Suppose that the scattered fields are asymptotically small relative to the pivots,
\[
  r(\tau)
  :=
  \frac{\sup_i\norm{U_i-(U_0)_i}_{L^2(\OmegaD)}}{\min_i\,(L_0)_{ii}}
  \longrightarrow 0
  \qquad\text{as }\tau\to0.
\]
Then $\left|\widehat{S}_{ii}-1\right|\le r(\tau)$ for every $i$, so for $\tau$ small enough the diagonal entries of $\widehat{S}$ are positive and $\widehat{U}$ is admissible.
\end{lemma}

\begin{proof}
Writing $(U_0)_i=\sum_{k\le i}(L_0)_{ik}\,\bar u^0_k$, the coefficient of $(U_0)_i$ in the sequential projection $\widehat{U}_i$ of $U_i$ is
\[
  \widehat{S}_{ii}
  =
  \frac{\ip{U_i}{\bar u^0_i}}{\ip{(U_0)_i}{\bar u^0_i}}
  =
  1+\frac{\ip{U_i-(U_0)_i}{\bar u^0_i}}{(L_0)_{ii}},
\]
and the claim follows from the Cauchy--Schwarz inequality, since $\norm{\bar u^0_i}_{L^2(\OmegaD)}=1$.
\end{proof}

\begin{remark}
The ratio $r(\tau)$ shows how singular sources and well chosen sampling drive the one dimensional proofs of lemmas 5.5 and 6.3 of \cite{DrMoZa6}. There the background snapshots are not bounded in $L^2$: they grow like $\tau^{-1/2}$, and the pivots inherit this growth.  In general the pivots can be much smaller than the norms if there is oversampling and the family becomes nearly linearly dependent. So, the denominator of $r(\tau)$ combines the concentration of the sources with the uniform independence of the snapshot family.  The scattered fields, on the other hand, are smoother than the background fields, because they arise from integration of the concentrated fields against the Green's function kernel.  In one dimension the kernel is bounded, so together with the $L^1$ normalization of the sources the scattered fields are uniformly bounded, so $r(\tau)=O(\sqrt{\tau})$; this is the sup norm bound of remark 5.3 of \cite{DrMoZa6}.  In higher dimensions the kernel is weakly singular, so the a priori bound on the scattered fields is weaker; the mechanism is the same, and as measured below the gain is if anything larger.  The pivots involve only the known background medium and are computable.
For the radial tent sources in the configuration of section~\ref{sec:numerics}, in the convention of section~\ref{subsec:bestapprox} (sampling at a quarter of the pulse width, $h=\tau/40$), over $\tau$ from $1/2$ to $1/8$ the smallest pivot grows like $\tau^{-1.5}$, while the scattered field norms $\sup_i\norm{U_i-(U_0)_i}_{L^2(\OmegaD)}$ remain bounded for both media, so $r(\tau)$ tends to zero at a rate close to $\tau^{1.5}$.  Consistently, in the experiments of section~\ref{subsec:bestapprox} the diagonal entries $\widehat{S}_{ii}$ are observed to lie in $[0.84,1]$ and to approach one under refinement, as in the one dimensional asymptotics.
\end{remark}

\section{Noisy data and regularization}
\label{sec:noise}
We model measured response data as
\begin{equation}
  \label{eq:noisy-response}
  F^\delta_{\ell m}(t_j)
  =
  F_{\ell m}(t_j)
  +
  \delta\,\sigma_F\,\eta_{\ell m j},
\end{equation}
where $\eta_{\ell m j}$ are independent standard normal variables and
\[
  \sigma_F =
  \left(
    \frac{1}{n_b^2 n_t}
    \sum_{\ell,m,j} |F_{\ell m}(t_j)|^2
  \right)^{1/2}
\]
is the root mean square response amplitude, with $n_t$ the number of time samples as in section~\ref{sec:construction}.  The noisy mass matrix $M^\delta$ is assembled from $F^\delta$ using \eqref{eq:block-gramian}.  Even if the continuum Gramian is positive, $M^\delta$ may be indefinite.  A direct Cholesky factorization then fails, or, if repaired only by a tiny shift, may amplify the noise through the small eigenvalues; see \cite{ChPaSt} for a perturbation analysis of the Cholesky factorization.

The reconstruction algorithm with noise is:
\begin{enumerate}
  \item acquire or simulate response samples $F^\delta_{\ell m}(j\tau)$;
  \item optionally symmetrize $F^\delta_{\ell m}\leftarrow(F^\delta_{\ell m}+F^\delta_{m\ell})/2$;
  \item assemble $M^\delta$ from \eqref{eq:block-gramian};
  \item replace $M^\delta$ by a positive regularized matrix $M^\delta_\alpha$;
  \item compute $\mathbf{U}_\alpha$ from \eqref{eq:reconstruction} with $M^\delta_\alpha$.
\end{enumerate}

We focus on two regularizations: a diagonal shift and an eigenvalue floor.  The first method minimally shifts all modes; the second is a spectral repair that more selectively suppresses unstable modes.  Truncation-based spectral regularizations of data-driven Gramians have been used previously in the inversion context, see \cite{BoDrMaZa3,baker2025regularized}.

\subsection{Diagonal shift}

The minimal repair is
\[
  M^\delta_\alpha = M^\delta + \alpha I,\qquad
  \alpha\ge \max\{0,-\lambda_{\min}(M^\delta)\}+\epsilon .
\]
This is a Tikhonov-type regularization \cite{engl1996regularization}; in the experiments below $\epsilon=10^{-10}$.  It guarantees a Cholesky factorization but does not selectively damp directions dominated by noise.

\subsection{Eigenvalue flooring}

Let $M^\delta=V\Lambda V^\top$.  The eigenvalue-floor reconstruction replaces
\[
  \Lambda_i
  \quad\text{by}\quad
  \max\{\Lambda_i,\lambda_{\mathrm{floor}}\},
  \qquad
  \lambda_{\mathrm{floor}}
  =
  \max\{\rho\,\lambda_{\max}(M^\delta),\epsilon\}.
\]
Here $\rho$ is a dimensionless relative floor, tying the threshold to the scale of the data Gramian, and $\epsilon$ is an absolute safeguard for the degenerate case in which $\lambda_{\max}$ itself is small or nonpositive.  In all experiments below we take $\rho=10^{-2}$ and $\epsilon=10^{-10}$; the choice is discussed in section~\ref{subsec:noise-study}.  This keeps the well-determined energetic modes nearly unchanged while preventing unstable small or negative modes from entering the Cholesky factor.  It is closely related to the computation of a nearest positive semidefinite matrix \cite{higham1988computing}.

\section{Numerical experiments}
\label{sec:numerics}
In this section we do numerical experiments where we test internal solution convergence empirically and consider noisy data.
In our studies here, we test with several different media.  The refinement studies of section~\ref{subsec:bestapprox} use a smooth Gaussian bump of amplitude $20$ and a discontinuous square ring of contrast $20$, both defined at the start of section~\ref{subsec:bestapprox}.  The examples of section~\ref{subsec:noise-study} use high contrast composite media, built from a Gaussian bump of amplitude $120$ with discontinuous inclusions of height $90$.
In all of the experiments we use a square domain $[0,1]^2$, homogeneous Neumann boundaries, localized radial tent source profiles near the boundary, and zero background $q_0=0$.  The response data and the true reference field are generated on a nested fine grid with twice the spatial resolution and half the time step of the reconstruction grid, so that the measured data carry a discretization error different from that of the background solves and an inverse crime is avoided.  The background snapshots are computed on the reconstruction grid, as they would be in practice.  Both solvers use a fourth order spatial discretization. The response sampling rate and the grid rule differ between the refinement studies and the noise study, and are stated in the corresponding subsections.

\subsection{Refinement studies and best approximation comparison}
\label{subsec:bestapprox}

In this subsection we examine how the reconstruction behaves as the number of sources, time sampling, and source radius are refined simultaneously.
Boundary sources are placed along all four sides, with their spacing tied to the pulse radius $\tau$: with $n_s$ intervals per side, the source count is $K=4n_s$ and the refinement parameter is $\tau=1/n_s$.  The radial tent pulses have radius $\tau$ with fixed unit $L^1$ norm, so that their height grows as $\tau$ decreases.  We emphasize that the two refinements play complementary roles. The time samples resolve the fields in the direction of propagation, into the domain, since consecutive background snapshots from a given source are wavefronts a fixed fraction of $\tau$ apart. The multiple sources resolve the fields laterally, along the boundary.  Tying the source spacing to $\tau$ refines both resolutions together, so that the background family covers the domain at roughly scale $\tau$; this is the two dimensional counterpart of the refinement in \cite{DrMoZa6}, where the lateral direction is absent.  In summary, refining $\tau$ simultaneously increases the number of sources, shortens the sampling interval, and sharpens the approximate boundary delta.

The response sampling in this subsection is at a quarter of the pulse width, $\Delta t_{\mathrm{sample}}=\tau/2$, and the reconstruction grid is tied to the pulse by $h=\tau/40$, so that the sharpening pulses remain equally resolved across the experiments.   In these sweeps $n_s=2,\ldots,8$, so $\tau$ ranges from $1/2$ to $1/8$.  We consider two media: a smooth Gaussian bump of amplitude $20$,
\[
  q_{\mathrm{g}20}(x,y)=20\exp\left(-\frac{(x-0.28)^2+(y-0.36)^2}{2(0.12)^2}\right),
\]
at target time $t\approx0.5$, and a discontinuous square ring of contrast $20$,
\[
  q_{\mathrm{r}20}(x,y)=20\,\mathbf{1}_{0.05\le \max(|x-0.30|,|y-0.50|)\le 0.10},
\]
at target time $t\approx0.34$.  Figures~\ref{fig:gaussian20-gallery} and~\ref{fig:ring20-gallery} show the two coefficients together with representative reconstructions at the finest refinement $\tau=1/8$ for a left-side source. Note that the error between the data generated and true field is far smaller and less structured than the background discrepancy.

\begin{figure}[ht]
\centering
\includegraphics[width=\linewidth]{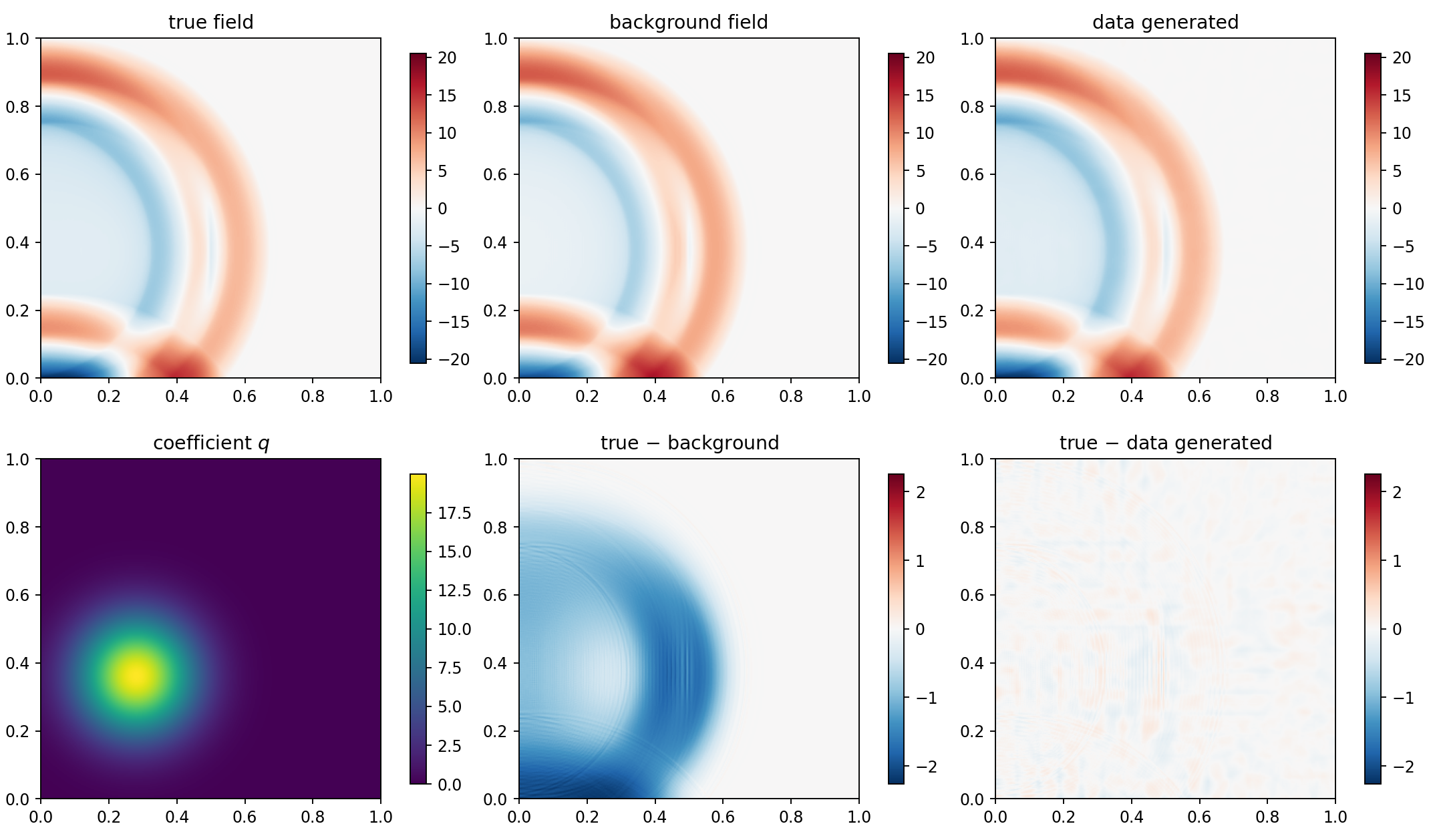}
\caption{The smooth medium of the refinement studies: Gaussian bump of amplitude $20$, $\tau=1/8$, left-side source. The panels show the coefficient $q_{\mathrm{g}20}$, the true field at the sample time nearest $t=0.5$, the background field, the data generated reconstruction, its error, and the error for the background field.  }
\label{fig:gaussian20-gallery}
\end{figure}

\begin{figure}[ht]
\centering
\includegraphics[width=\linewidth]{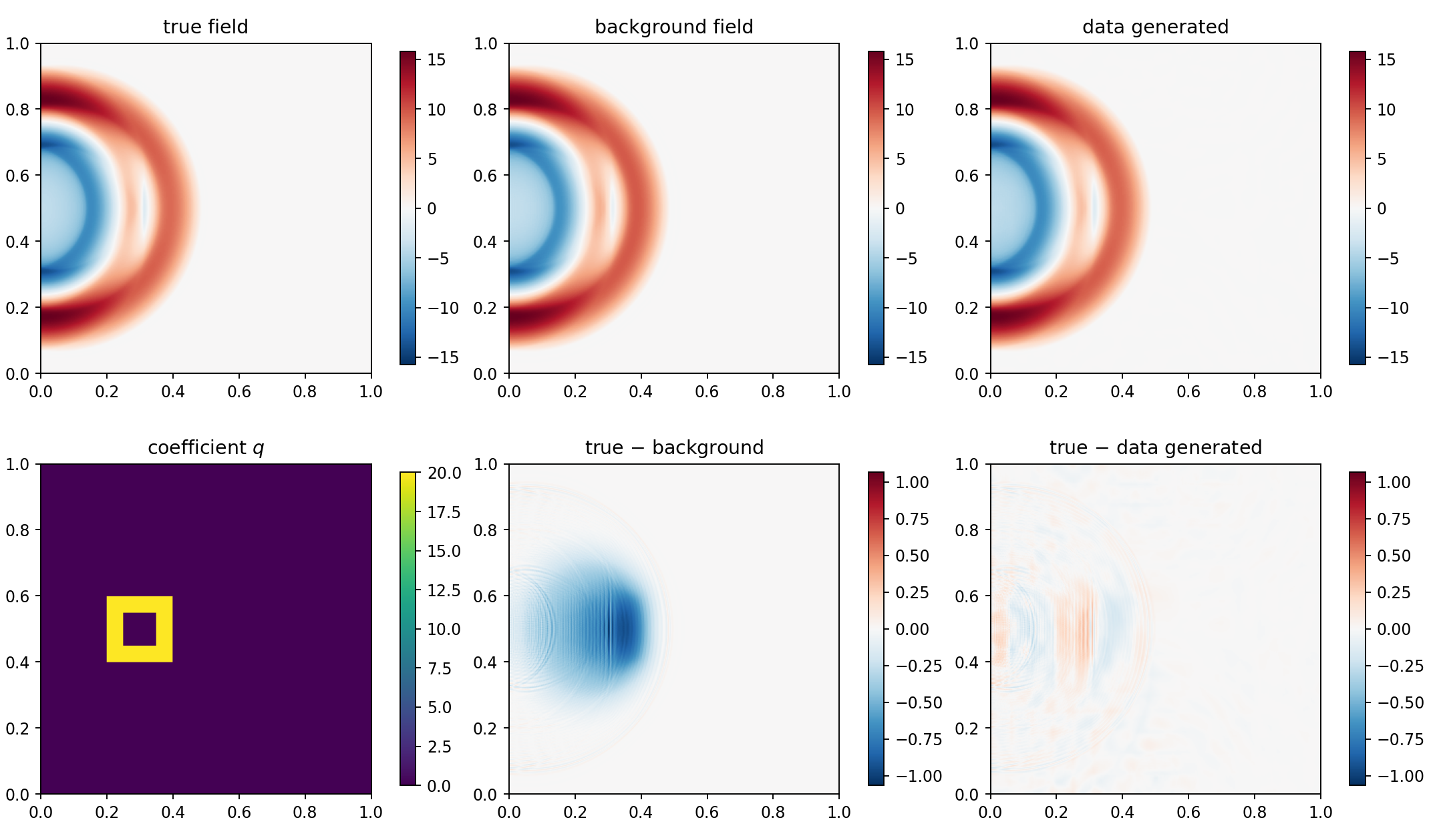}
\caption{The discontinuous medium of the refinement studies: square ring of contrast $20$ centered at $(0.30,0.50)$, $\tau=1/8$, left-side source.  The panels show the coefficient $q_{\mathrm{r}20}$, the true field at the sample time nearest the target $t=0.34$ (here $t=0.3125$), the background field, the data generated reconstruction,  its error, and the error for the background field.  }
\label{fig:ring20-gallery}
\end{figure}

To measure the error we use the aggregate relative error in the stacked $L^2$ norm \eqref{eq:stacked-norm},
\[
  E_{\mathrm{agg}}
  =
  \frac{\norm{U-\mathbf{U}}_2}{\norm{U}_2},
\]
which is closely related to---though not identical with---the quantity controlled by Proposition~\ref{prop:general-bound}. The proposition bounds the distance from $\mathbf{U}$ to the best causal approximation $\widehat{U}$, whereas $E_{\mathrm{agg}}$ measures the distance to the true field $U$.  For the solution figures and the noise study we also use the per-snapshot relative error and the ratio to the background error (using the continuous notation $u_{\mathrm{true}}=U$, ${\bf u}={\bf U}$, $u_0=U_0$),
\[
  E_{\mathrm{rel}}
  =
  \frac{\norm{u_{\mathrm{true}}(\cdot,t)-{\bf u}(\cdot,t)}_2}
       {\norm{u_{\mathrm{true}}(\cdot,t)}_2},
  \qquad
  E_{\mathrm{DG/BG}}
  =
  \frac{\norm{u_{\mathrm{true}}-{\bf u}}_2}
       {\norm{u_{\mathrm{true}}-u_0}_2},
\]
where values of $E_{\mathrm{DG/BG}}$ below one mean that the data generated reconstruction improves on the background field.

The analysis in \cite{DrMoZa6} shows that the data generated internal solutions are asymptotically close to the causal projection of the true snapshots onto the space of background snapshots, $\widehat{U}$.  We therefore also measure the best causal approximation error, in addition to the full $L^2$ best approximation error, which does not enforce triangularity.    The gap between the two variants (causal projection and full projection) measures how much the triangular structure costs. It is important to mention that the projection errors are smaller than the data generated errors and are therefore more sensitive to the discretization error.  In a grid-doubling experiment at $n_s=4$ the data generated errors are essentially unchanged, while the projection errors drop substantially: at $h=\tau/40$ the projections are dominated by the representation error of the grid rather than by the approximation properties of the family.  Because of this we plot the data generated and background errors at $h=\tau/40$, while the projection curves are computed at $h=\tau/80$ and reported only for $n_s\le5$. 

In Figure~\ref{fig:refined-convergence} we plot the aggregate relative errors over the refinement.   We emphasize that the convergence rate of $\sqrt{\tau}$ of \cite{DrMoZa6} has not been established in two dimensions, so we measure the empirical rates and show the $\sqrt{\tau}$ line only as a reference.  For the smooth Gaussian medium, the decay is faster than first order at the coarse resolutions and settles onto a rate consistent with the $\sqrt{\tau}$ reference at the finest ones, while the resolved best approximation errors decay at close to first order.  For the discontinuous ring the errors are smaller in absolute terms but the decay is slower, and beyond $n_s=6$ the error is dominated by the grid and omitted.

It is important to note that the error in Figure~\ref{fig:refined-convergence} is relative. Since the normalized sources concentrate as $\tau$ decreases, the $L^2$ norms of the fields in the denominators grow. In Figure~\ref{fig:l1-absolute} we therefore show the same refinement in the absolute per-snapshot mean $L^1$ error. We do this on the resolved grid $h=\tau/80$, for the background field, the data generated reconstruction, and the best causal approximation.   In this norm the data generated error decreases at fitted log-log rates of $0.42$ and $0.46$, while in the scaled absolute $L^2$ norm they level off. 

It is not clear that in these experiments the relationship between the time step, number of sources, and source concentration is optimal. Refinement rules with better convergence scaling, together with a rigorous best approximation analysis of the two dimensional family, are the subject of future work.
\begin{figure}[ht]
\centering
\includegraphics[width=\linewidth]{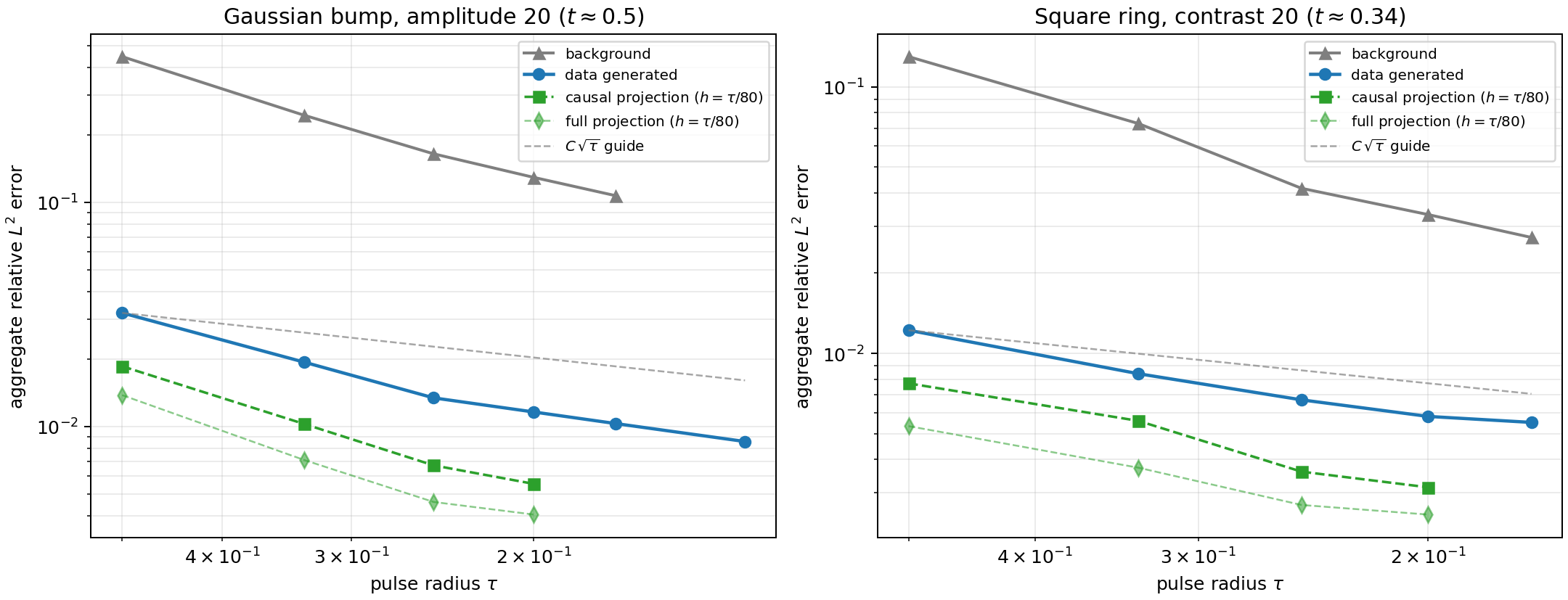}
\caption{Refinement studies with response sampling at a quarter of the pulse width ($\Delta t_{\mathrm{sample}}=\tau/2$) and grid $h=\tau/40$.  Left: Gaussian bump of amplitude $20$ at $t\approx0.5$; right: square ring of contrast $20$ at $t\approx0.34$.  Each panel shows the aggregate background error, the aggregate data generated error, and the causal and full projection errors (the projections computed at $h=\tau/80$, $n_s\le5$); the $C\sqrt{\tau}$ line is the one dimensional reference rate of \cite{DrMoZa6}, not an established rate in two dimensions.  Only resolution-verified points are shown: the square-ring curves end at $n_s=6$, beyond which the measured values sit at the representation floor of the $h=\tau/40$ grid rule.}
\label{fig:refined-convergence}
\end{figure}

\begin{figure}[ht]
\centering
\includegraphics[width=\linewidth]{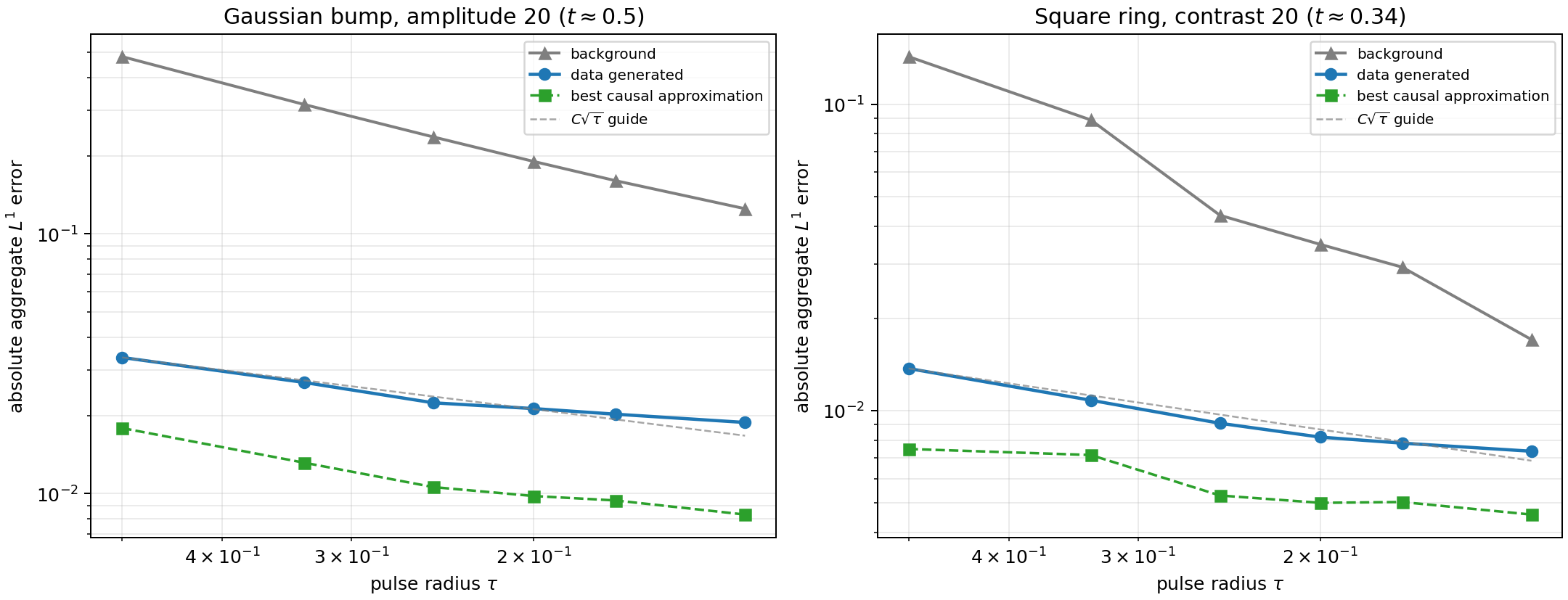}
\caption{The refinement of Figure~\ref{fig:refined-convergence} measured in the absolute aggregate $L^1$ error, averaged over the stacked snapshots, on the resolved grid $h=\tau/80$: the background field, the data generated reconstruction, and the best causal approximation, for the Gaussian bump (left) and the square ring (right).  In contrast with Figure~\ref{fig:refined-convergence}, which shows relative errors, no normalization by the growing snapshot norms is involved here.  The data generated error decreases monotonically at rates $0.42$ and $0.46$, close to the $\sqrt{\tau}$ reference, and stays within a constant factor of the best causal approximation at every refinement.  All six refinements of the sweep are shown here, since on this grid every point is resolution verified, whereas the curves of Figure~\ref{fig:refined-convergence} are restricted to the points verified on the coarser rule.}
\label{fig:l1-absolute}
\end{figure}

\subsection{High contrast composite media and noisy data}

\label{subsec:noise-study}

The experiments in this subsection use the earlier convention of \cite{DrMoZa6}, with response sampling at half the pulse width, $\Delta t_{\mathrm{sample}}=\tau$, and reconstruction grid $h=\tau/20$.  The fields are visually indistinguishable between the two conventions, and the qualitative comparisons below do not depend on the choice.  The single exception is the strong scattering example of Figure~\ref{fig:strong-contrast}, where the contrast is high enough that the sampling rate is visible in the reconstruction; that figure is computed with the finer time sampling of $\Delta t_{\mathrm{sample}}=\tau/2$.  To test how the solution reconstructions hold for more complex media,  we use a composite medium built from the higher contrast Gaussian bump
\[
  q_{\mathrm{g}}(x,y)=120\exp\left(-\frac{(x-0.28)^2+(y-0.36)^2}{2(0.12)^2}\right)
\]
with a square-ring inclusion,
\[
  q_{\mathrm{ring}}(x,y)
  =
  q_{\mathrm{g}}(x,y)
  +
  90\,\mathbf{1}_{0.06\le \max(|x-0.72|,|y-0.72|)\le 0.12}.
\]
This medium introduces sharp interfaces and spatially separated scattering features, and is also the medium of the noise study.  Figure~\ref{fig:square-ring} shows the reconstruction at the finest refinement $\tau=1/8$, for a left-side source at $t\approx0.5$.  In this example, the data generated field captures the phase and amplitude changes induced by the inclusion more accurately than the unperturbed background field.  The remaining error is concentrated near the scattered wavefronts and around the sharp interfaces. Recall that the background snapshots are exactly in the approximation space, so the error depends on the regularity of the scattered field \cite{DrMoZa6}.

\begin{figure}[ht]
\centering
\includegraphics[width=\linewidth]{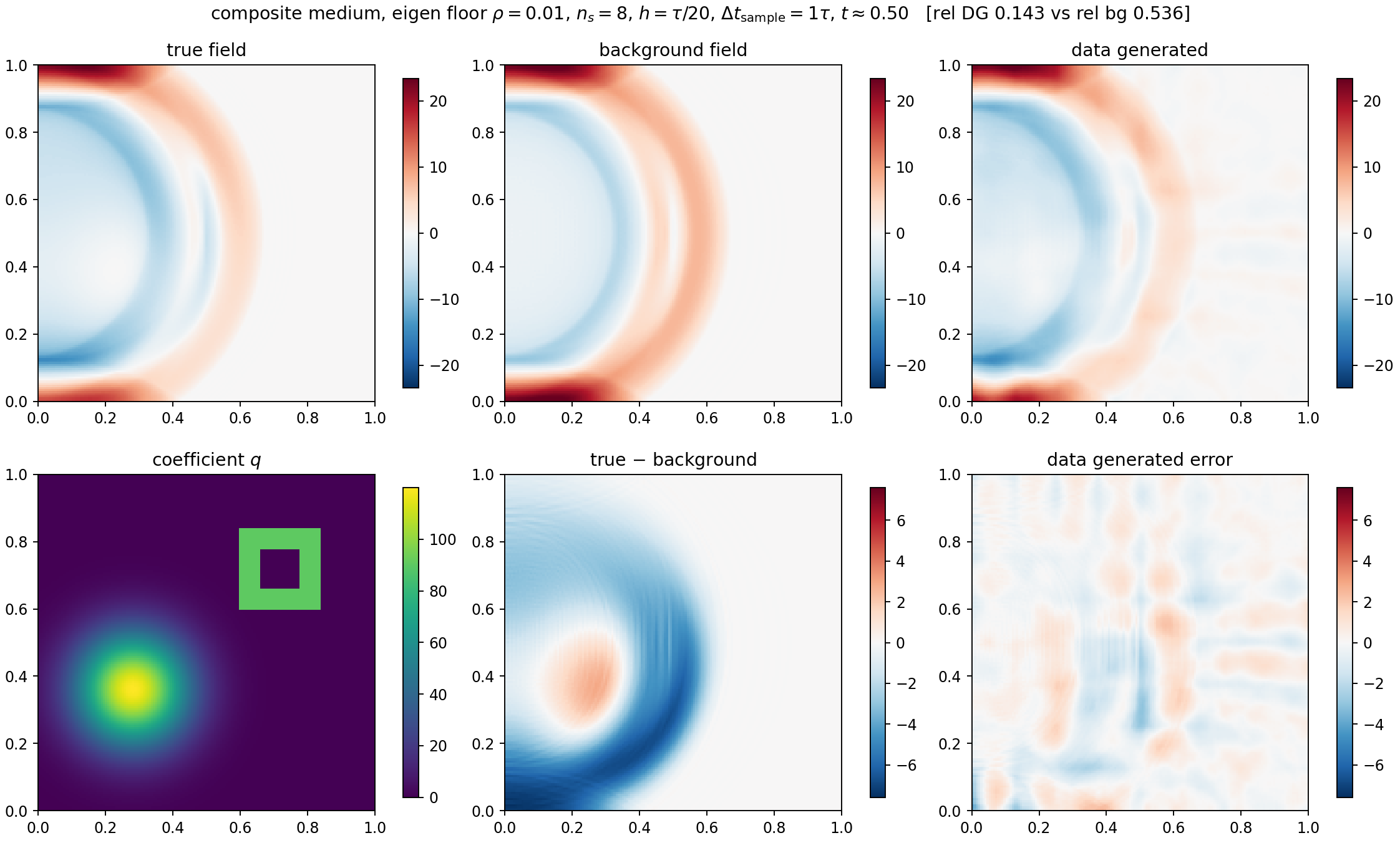}
\caption{Composite medium with a Gaussian bump and a square-ring inclusion: the coefficient $q_{\mathrm{ring}}$, the true and background fields at $t\approx0.5$, the data generated reconstruction, and the two residuals on a common color scale. Here $\tau=1/8$ and the source is from the left.  The reconstruction uses the eigenvalue floor at $\rho=10^{-2}$, as in the noise study below, and has relative error $14.3\%$. The background solution has error $53.6\%$.}
\label{fig:square-ring}
\end{figure}

 We also consider pushing the contrast even further. Figure~\ref{fig:strong-contrast} shows a reconstruction of a Gaussian bump of amplitude $240$, where we take $n_s=10$, with response sampling at a quarter of the pulse width as in section~\ref{subsec:bestapprox}.  At this contrast the sampling rate limits the reconstruction: on the grid $h=\tau/20$ used throughout this subsection, halving the sampling interval from $\Delta t_{\mathrm{sample}}=\tau$ to $\tau/2$ lowers the relative error of the data generated solution from $24\%$ to $11\%$ and removes the granularity of the error panel, while the background discrepancy is unchanged at $84\%$.
 
\begin{figure}[ht]
\centering
\includegraphics[width=\linewidth]{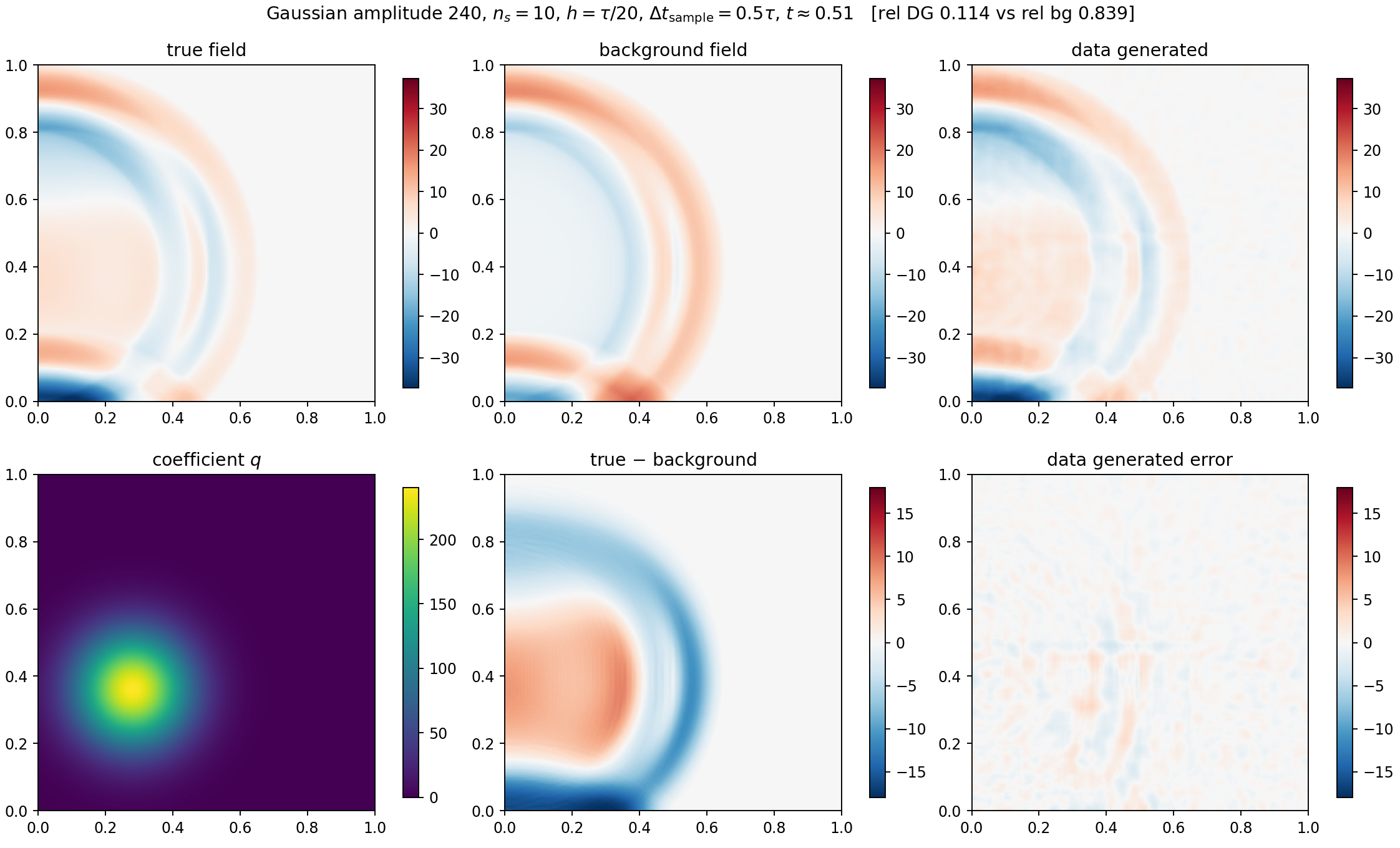}
\caption{Strong scattering: the Gaussian medium with amplitude $240$, $n_s=10$, at $t\approx0.5$, with response sampling at a quarter of the pulse width ($\Delta t_{\mathrm{sample}}=\tau/2$) on the reconstruction grid $h=\tau/20$.  Top row: true field, background field, and data generated reconstruction.  Bottom row: the coefficient $q$, the difference between the true and background fields, and the data generated error, on a common color scale.  At this contrast the background field has relative error $84\%$, while the data generated solution has relative error $11\%$, with the error concentrated in and behind the scatterer; the restructured wavefront is recovered.}
\label{fig:strong-contrast}
\end{figure}
The noise study uses the square-ring composite medium $q_{\mathrm{ring}}$ defined above (Figure~\ref{fig:square-ring}), in the convention stated at the start of this subsection, at the finest refinement $\tau=1/8$: fine-grid response data, fourth order solvers, sampling at half the pulse width, and $K=32$ boundary sources on the $161^2$ reconstruction grid.  The target is the snapshot near $t=0.5$ for a left-side source, for which the data generated error from clean data is $14.3\%$ against a background discrepancy of $53.6\%$.  Noisy data are generated from \eqref{eq:noisy-response}, and results are means over three noise realizations.

Figure~\ref{fig:noise-summary} summarizes the study, showing the relative error and the DG/BG ratio as functions of the noise level for the two regularizations of section~\ref{sec:noise}.  The eigenvalue floor is strikingly insensitive to the noise: the reconstruction error rises only from $14\%$ on clean data to $15\%$ at $5\%$ response noise and $19\%$ at $10\%$, so the reconstruction remains about three times more accurate than the unperturbed background field across the whole range.  The comparison between the two regularizations makes the case for the floor.  The diagonal shift stabilizes the Cholesky factorization, but it retains the noise amplification of the small eigenvalues and degrades steadily: at $5\%$ noise its reconstruction is no better than the background field, and at $10\%$ it is worse.

The strength of the floor matters, and Figure~\ref{fig:noise-floor} shows why we take $\rho=10^{-2}$.  Below about $10^{-4}$ the floor touches no mode of the data Gramian and the reconstruction is effectively unregularized, so the error is the same as with no floor at all; above about $10^{-1}$ features are lost, ane we were able to use the single $\rho$ value for all the experiments. 
Figure~\ref{fig:noise-gallery} shows the true and background fields side by side with the noise free and noisy reconstructions. 

\begin{figure}[p]
\centering
\includegraphics[width=\linewidth]{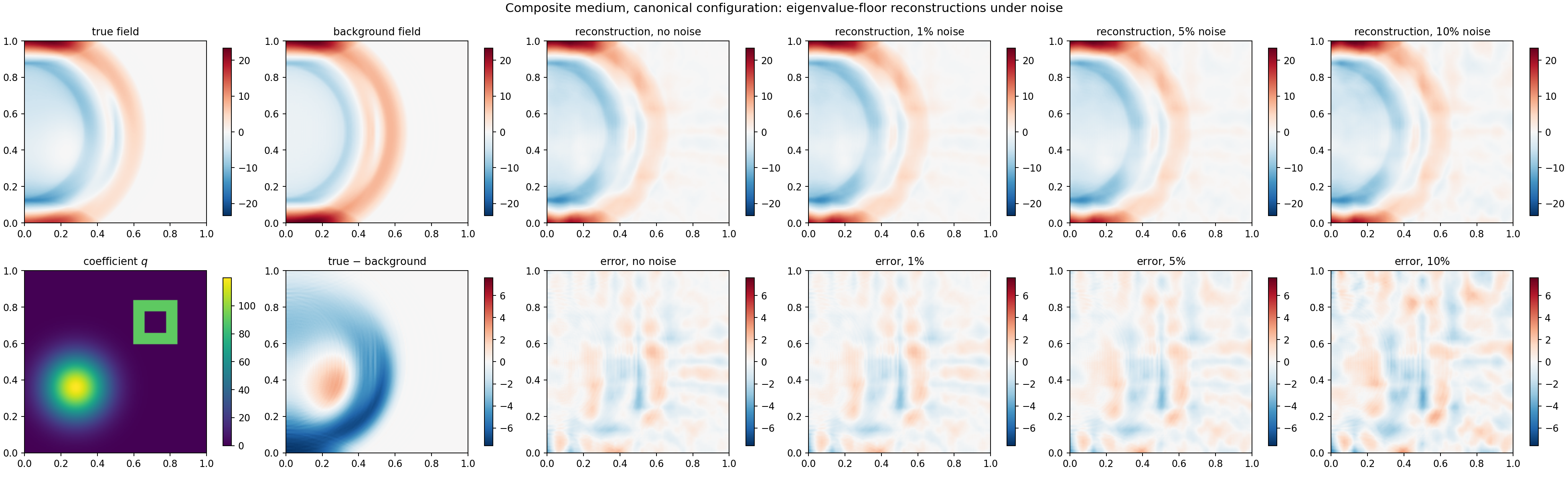}
\caption{Noise study for the composite medium.  Top row: the true field at $t\approx0.5$, the unperturbed background field, and the eigenvalue-floor reconstructions from noise free data and from data with $1\%$, $5\%$, and $10\%$ response noise.  Bottom row: the coefficient $q$, the background error, and the corresponding reconstruction errors, all on a common scale.  The background discrepancy contains the coherent scattered wave, while the reconstruction errors remain largely unstructured.}
\label{fig:noise-gallery}
\end{figure}

\begin{figure}[ht]
\centering
\includegraphics[width=0.86\linewidth]{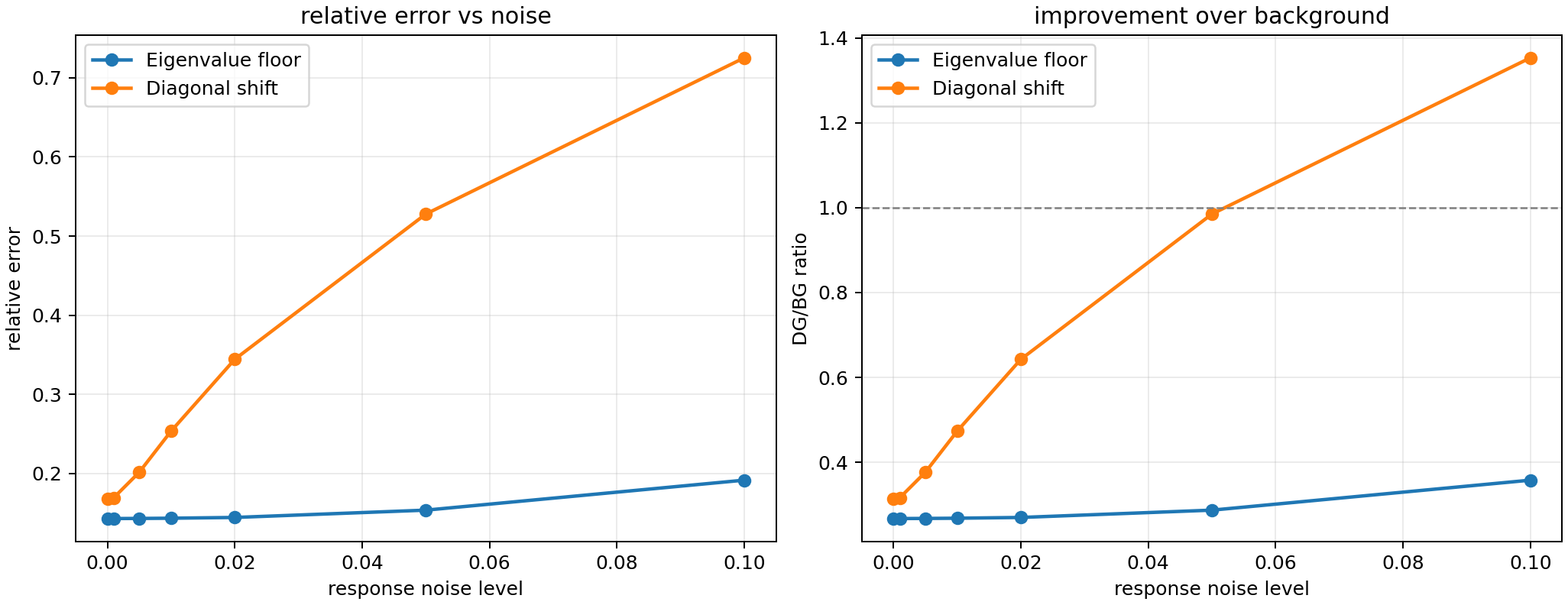}
\caption{Relative error and DG/BG ratio as functions of the response noise level for the diagonal shift and eigenvalue-floor regularizations.  Values above the dashed line mean that the reconstruction is less accurate than the unperturbed background field.}
\label{fig:noise-summary}
\end{figure}

\begin{figure}[ht]
\centering
\includegraphics[width=0.62\linewidth]{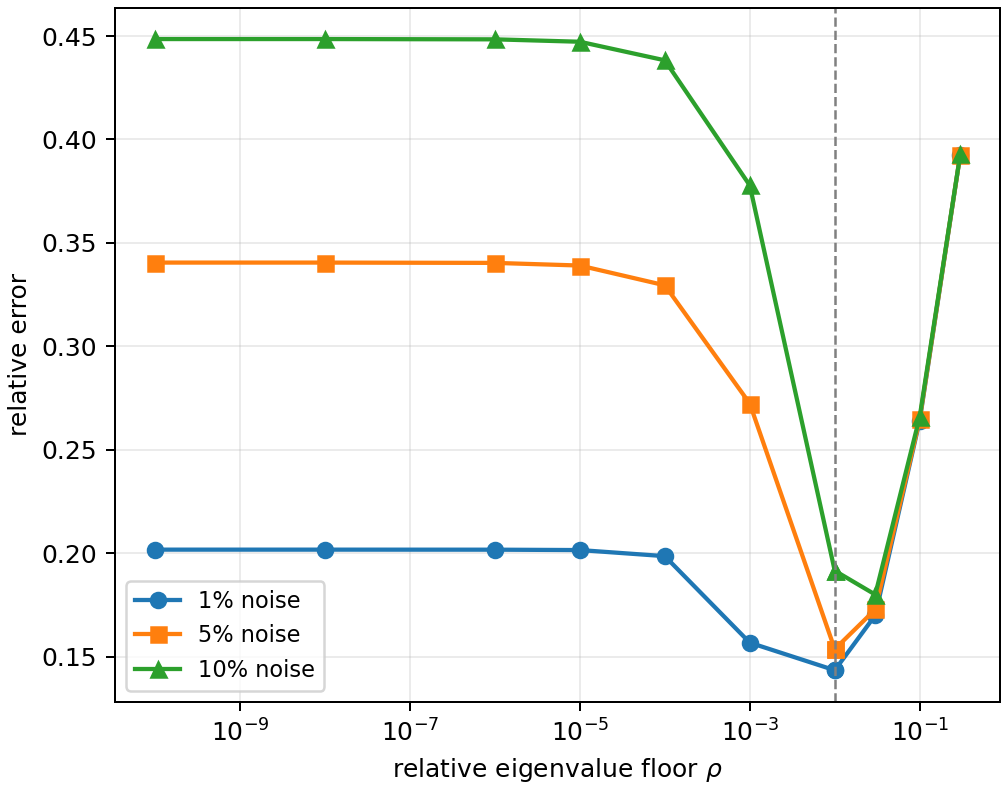}
\caption{Relative error of the eigenvalue-floor reconstruction as a function of the relative floor $\rho$ at three noise levels. The dashed line marks the value $\rho=10^{-2}$ used throughout.}
\label{fig:noise-floor}
\end{figure}

\section{Discussion}
\label{sec:discussion}

The experiments demonstrate that one can compute internal wave fields from boundary response data  without any knowledge of $q$. We see that multiple boundary sources are important in two dimensions: they enrich the snapshot space and make the block Gramian a more informative representation of interior propagation, as in the discussion of the MIMO setting in \cite{DrMoZa6}. Setups with rigorous convergence rates have not been established, but the numerical evidence suggests that $L^1$ convergence is possible.  We also show that the field reconstructions are not too sensitive to noise if one regularizes the Gramian with eigenvalue flooring.  Eigenvalue flooring preserves the dominant data-supported subspace while suppressing the modes most affected by noise and discretization.

The media considered here have all had strong enough contrast that the scattered field remained well above the noise; we have not considered the low contrast regime, where the noise can easily exceed the perturbation signal and the background field itself becomes the better estimate.  A regularization that detects this situation and reverts to the background solution automatically is the subject of future work.

The generation of internal fields from boundary data has the  limitation that it reconstructs snapshots in the (causal) span of the background wavefields. For wave speed problems where the unknown coefficient produces fields that are not triangular, caustics, or other phenomena poorly represented by the background snapshot family, the reconstruction will be biased. 

Future work also includes a rigorous convergence analysis of the MIMO formulation extending \cite{DrMoZa6}, including a best approximation analysis of the two dimensional background snapshot family and refinement rules with improved convergence scaling, and further use of the reconstructed internal fields in the Lippmann--Schwinger equation \cite{DrMoZa,BoGaMaZi} to recover the plasma coefficient $q$.  

\section*{Acknowledgments}

The authors used AI assistants (Anthropic's Claude and OpenAI's ChatGPT) to help draft and edit the manuscript and to write and run the numerical experiment code under their direction.  The authors have verified all results and take full responsibility for the content.

V. Druskin was partially supported by the AFOSR grant FA 955020-1-0079 and the NSF grant  DMS-2110773. M. Zaslavsky was partially supported by  AFOSR grant  FA9550-20-1-0079. S. Moskow was partially supported by NSF grant DMS-2308200.  

\section*{Reproducibility}

The complete code for the experiments in this paper is available at
\begin{center}
  \url{https://github.com/sharimoskow/plasma-internal-solutions-2d}
\end{center}
The implementation is self-contained Python with no external PDE libraries.  All of the drivers below take the sampling rate, grid rule, solver order, pulse shape, medium, contrast, and regularization as options, so each experiment can be repeated under a different convention.

The refinement studies of section~\ref{subsec:bestapprox} are produced by \path{study_best_approximation_tau_sweep.py} (sampling at a quarter of the pulse width, fine-grid data, fourth order solvers, grid rule $h=\tau/40$), with the convergence figure of Figure~\ref{fig:refined-convergence}, restricted to resolution-verified points, drawn by \path{resolved_convergence_figs.py}.  The absolute $L^1$ comparison of Figure~\ref{fig:l1-absolute} is produced by \path{hc_paper_sweep.py} on the resolved grid $h=\tau/80$, which records the background, data generated and projection errors in both norms from the same runs, and is drawn by \path{l1_absolute_paper_fig.py}.  The quantities entering the ratio $r(\tau)$ of Lemma~\ref{lem:positivity}, the background pivots and the scattered field norms, are computed by \path{check_positivity_newconv.py} (results in \path{positivity_newconv.csv}); the same quantities, together with the diagonal of $\widehat{S}$, are recorded by \path{hc_paper_sweep.py}.

For section~\ref{subsec:noise-study}, the strong scattering figure is produced by \path{strong_contrast_fig.py}, which takes the sampling rate as an option and was run at $\Delta t_{\mathrm{sample}}=\tau/2$ for Figure~\ref{fig:strong-contrast}.  The forward arrays for the noise study are generated once by \path{make_noise_npz.py} and replayed by \path{noise_canonical_study.py}, which performs no further wave solves: the noise realizations, the two regularizations, and the summary and gallery figures all follow from the stored responses.  The floor sweep of Figure~\ref{fig:noise-floor} is produced by \path{floor_sweep.py} from the same arrays.  The repository README maps each figure to the exact commands.

\bibliographystyle{plain}
\bibliography{ROMbibliography,plasma_boundary_refs}

\end{document}